\documentclass[12pt,twoside]{amsart}
\usepackage{amssymb,amsmath,amsthm, amscd, enumerate, mathrsfs}
\usepackage{graphicx, hhline}
\usepackage[all]{xy}
\usepackage[usenames]{color}
\usepackage{hyperref}
\hypersetup{colorlinks=true}

\title{On log canonical rational singularities}
\author{Osamu Fujino} 
\date{2015/3/4, version 1.00}
\address{Department of Mathematics, Graduate School of Science, 
Kyoto University, Kyoto 606-8502, Japan}
\email{fujino@math.kyoto-u.ac.jp}

\newcommand{\Exc}[0]{{\operatorname{Exc}}}
\newcommand{\Supp}[0]{{\operatorname{Supp}}}
\newtheorem{thm}{Theorem}[section]
\newtheorem{lem}[thm]{Lemma}
\newtheorem{cor}[thm]{Corollary}

\theoremstyle{definition}

\newtheorem{rem}[thm]{Remark}
\newtheorem*{ack}{Acknowledgments} 
\newtheorem{say}[thm]{}
\newtheorem{ex}[thm]{Example}

\newtheorem{case}{Case}
\begin{document}

\maketitle 

\begin{abstract}
We prove that the class of log canonical rational singularities is closed 
under the basic operations of the minimal model program. 
We also give some supplementary results on the minimal model program 
for log canonical surfaces. 
\end{abstract}

\tableofcontents 

\section{Introduction}\label{sec1} 

In this short note, we prove the following theorems, which are missing 
in \cite{fujino-fund}. This short note is a supplement to \cite{fujino-fund}, 
\cite{fujino-some}, and \cite{fujino-surface}.  

\begin{thm}\label{thm1.1} 
Let $(X, \Delta)$ be a log canonical 
pair and let $f:X\to Y$ be a projective 
surjective morphism such that 
$f_*\mathcal O_X\simeq \mathcal O_Y$ and that $-(K_X+\Delta)$ is $f$-ample. 
Assume that $X$ has only rational singularities. 
Then $Y$ has only rational singularities.  
\end{thm}

We can easily prove Theorem \ref{thm1.1} by 
the relative Kodaira type vanishing theorem for log canonical pairs 
and Kov\'acs's characterization of rational singularities. 
Of course, the vanishing theorem for log canonical pairs 
is nontrivial in the classical minimal model program (see \cite{kollar-mori}). 
However, now we can freely use such a powerful vanishing theorem 
for log canonical pairs 
(see, for example, \cite{fujino-fund} and \cite{fujino-foundation}). 
Note that we do not assume that $f$ is birational in Theorem \ref{thm1.1}. 

\begin{thm}\label{thm1.2}
We consider a commutative diagram 
$$
\xymatrix{
X \ar@{-->}[rr]^\phi\ar[rd]_f&& X^+ \ar[ld]^{f^+}\\ 
& Y&
}
$$
where $(X, \Delta)$ and $(X^+, \Delta^+)$ are log canonical, 
$f$ and $f^+$ are projective birational morphisms, 
and $Y$ is normal. 
Assume that 
\begin{itemize}
\item[(i)] $f_*\Delta=f^+_*\Delta^+$, 
\item[(ii)] $-(K_X+\Delta)$ is $f$-ample, and 
\item[(iii)] $K_{X^+}+\Delta^+$ is $f^+$-ample. 
\end{itemize}
We further assume that 
$X$ has only rational singularities. 
Then 
$X^+$ has only rational singularities. 
\end{thm}

Theorem \ref{thm1.2} follows from the well-known 
negativity lemma (see, for example, \cite[Lemma 3.38]{kollar-mori} 
and \cite[Lemma 2.3.27]{fujino-foundation}) and 
the result on nonrational centers of log canonical 
pairs due to Alexeev--Hacon (see \cite{alexeev-hacon}), which can be obtained 
in the framework of \cite{fujino-fund}. 

\begin{rem}\label{rem1.3}
In Theorem \ref{thm1.2}, the log canonicity 
of $(X^+, \Delta^+)$ follows from the 
other conditions of Theorem \ref{thm1.2} by the 
negativity lemma (see, for example, \cite[Lemma 3.38]{kollar-mori} and 
\cite[Lemma 2.3.27]{fujino-foundation}). 
It is sufficient to assume that $X^+$ is a normal variety and $\Delta^+$ is 
an effective $\mathbb R$-divisor on $X^+$ such that 
$K_{X^+}+\Delta^+$ is $\mathbb R$-Cartier. 
\end{rem}

Note that the singularities of $X$ are not 
always rational when $(X, \Delta)$ is only log canonical. 
Moreover, $X$ is not necessarily Cohen--Macaulay. 
This is one of difficulties when we treat log canonical 
pairs. We hope that Theorem \ref{thm1.1} and Theorem \ref{thm1.2} 
will be useful for the study of log canonical pairs. 

\begin{say}[MMP for log canonical pairs with only rational 
singularities]\label{say1.4}
Let us discuss the minimal model 
program for log canonical 
pairs with only rational singularities. 

Let $(X, \Delta)$ be a log canonical pair and let 
$\pi:X\to S$ be a projective 
morphism onto a variety $S$. 
Then we know that 
we can always run the minimal model 
program 
starting from $\pi:(X, \Delta)\to S$ (for the details, 
see, for example, \cite{fujino-fund}, 
\cite{birkar}, \cite{hacon-xu}, \cite{fujino-some}, \cite{fujino-foundation}, and so on). 
We further assume that $X$ has only 
rational singularities. 
Then, Theorem \ref{thm1.1} and Theorem \ref{thm1.2} say that 
every variety appearing in the minimal model program 
starting from $\pi:(X, \Delta)\to S$ has only 
rational singularities. 

From now on, we will see a contraction morphism more precisely. 
Let 
$$
f:(X, \Delta)\to Y
$$ 
be a contraction morphism  
such that 
\begin{itemize}
\item[(i)] $(X, \Delta)$ is a $\mathbb Q$-factorial log canonical 
pair, 
\item[(ii)] $-(K_X+\Delta)$ is $f$-ample, and  
\item[(iii)] $\rho(X/Y)=1$. 
\end{itemize}
Then we have the following three cases. 

\begin{case}[Divisorial contraction]\label{case1} 
$f$ is divisorial, that is, $f$ is a birational contraction which contracts 
a divisor. 
In this case, the exceptional locus $\Exc(f)$ of $f$ is a prime divisor 
on $X$ and $(Y, \Delta_Y)$ is a $\mathbb Q$-factorial log canonical pair with 
$\Delta_Y=f_*\Delta$. 
Moreover, if $X$ has only rational singularities, 
then $Y$ has only rational singularities by Theorem \ref{thm1.1}. 
\end{case}

\begin{case}[Flipping contraction]\label{case2} 
$f$ is flipping, that is, $f$ is a birational contraction which is small. 
In this case, we can take the flipping diagram: 
$$
\xymatrix{
X\ar@{-->}[rr]^{\varphi} \ar[rd]_{f} && X^+\ar[ld]^{f^+} \\
  & Y & 
} 
$$
where $f^+$ is a small projective birational morphism and  
\begin{itemize}
\item[(i$'$)] $(X^+, \Delta^+)$ is a $\mathbb Q$-factorial 
log canonical pair with $\Delta^+=\varphi_*\Delta$, 
\item[(ii$'$)] $K_{X^+}+\Delta^+$ is $f^+$-ample, and 
\item[(iii$'$)] $\rho(X^+/Y)=1$.  
\end{itemize}
By Theorem \ref{thm1.2}, we see that $X^+$ has only rational singularities 
when $X$ has only rational singularities. 
For the existence of log canonical flips, see \cite[Corollary 1.2]{birkar} and 
\cite[Corollary 1.8]{hacon-xu}. 
\end{case}

\begin{case}[Fano contraction]\label{case3}  
$f$ is a Fano contraction, that is, $\dim Y<\dim X$. 
Then $Y$ is $\mathbb Q$-factorial and has only log canonical 
singularities by \cite{fujino-some}. 
Moreover, if $X$ has only rational singularities, 
then $Y$ has only rational singularities by Theorem \ref{thm1.1}.  
\end{case} 

Anyway, the class of $\mathbb Q$-factorial log canonical 
rational singularities is closed under the minimal model 
program. 

Let $(X, \Delta)$ be a projective log canonical pair such that 
$K_X+\Delta$ is a semiample big $\mathbb Q$-Cartier divisor. 
Unfortunately, the log canonical model of $(X, \Delta)$ may have nonrational 
singularities even when $X$ has only rational singularities (see 
Example \ref{ex5.1}). This causes some undesirable phenomena 
(see Example \ref{ex5.3}). 
\end{say}

In this paper, we also give some supplementary 
results on the minimal model program for (not necessarily 
$\mathbb Q$-factorial) log canonical surfaces. 
We have: 

\begin{thm}[see Theorem \ref{thm4.1}]\label{thm1.5}
Let $(X, \Delta)$ be a log canonical 
surface and let $f:X\to Y$ be a projective 
birational morphism 
onto a normal surface $Y$. Assume that 
$-(K_X+\Delta)$ is $f$-ample. 
Then the exceptional locus $\Exc(f)$ of $f$ 
passes through no nonrational 
singular points of $X$. 
\end{thm}

By Theorem \ref{thm1.5}, the minimal model program for 
log canonical surfaces discussed in \cite[Theorem 3.3]{fujino-surface} 
becomes independent of the classification of 
numerically lc surface singularities in \cite[Theorem 4.7]{kollar-mori} 
(see Remark \ref{rem4.4}). 
When a considered surface is not $\mathbb Q$-factorial, 
the original proof of \cite[Theorem 3.3]{fujino-surface} 
uses the fact that a numerically lc surface is a log canonical surface (see 
\cite[Proposition 3.5 (2)]{fujino-surface}). 
For the proof of this fact, we need a rough classification of 
numerically lc surface singularities in \cite[Theorem 4.7]{kollar-mori} 
(see the proof of \cite[Proposition 3.5 (2)]{fujino-surface}). 

\begin{ack}
The author was partially supported by Grant-in-Aid for 
Young Scientists (A) 24684002 from JSPS. 
\end{ack}

We will work over $\mathbb C$, the complex number field, throughout this 
short note. 
We will freely use the basic notation of the minimal model 
program as in \cite{fujino-fund}. 

\section{Preliminaries}\label{sec2}

Let us recall the notion of singularities of pairs. 
For the details, see \cite{fujino-fund}, \cite{fujino-foundation}, and so on. 

\begin{say}[Singularities of pairs]\label{say2.1} 
A pair $(X, \Delta)$ consists of a normal variety $X$ 
and an effective $\mathbb R$-divisor 
$\Delta$ on $X$ such that $K_X+\Delta$ is $\mathbb R$-Cartier. 
A pair $(X, \Delta)$ is called 
kawamata log terminal (resp.~log canonical) if for any 
projective birational morphism 
$f:Y\to X$ from a normal variety $Y$, 
$a(E, X, \Delta)>-1$ (resp.~$\geq -1$) for 
every $E$, where 
$$K_Y=f^*(K_X+\Delta)+\sum _E a(E, X, \Delta)E.$$ 
Let $(X, \Delta)$ be a log canonical pair and let $W$ be a 
closed subset of $X$. Then $W$ is called a log canonical 
center of $(X, \Delta)$ if 
there are a projective birational morphism $f:Y\to X$ from a normal 
variety $Y$ and a prime divisor $E$ on $Y$ 
such that $a(E, X, \Delta)=-1$ 
and 
that $f(E)=W$. 
Let $(X, \Delta)$ be a log canonical pair. 
If there exists a projective birational morphism $f:Y\to X$ from a smooth 
variety $Y$ such that 
the $f$-exceptional locus 
$\Exc(f)$ and $\Exc (f)\cup \Supp f^{-1}_*\Delta$ are simple 
normal crossing divisors on $Y$ and that 
$a(E, X, \Delta)>-1$ for every $f$-exceptional divisor $E$, then 
$(X, \Delta)$ is called a divisorial log terminal pair. 
\end{say}

For surfaces, we can define $a(E, X, \Delta)$ without assuming 
that $K_X+\Delta$ is $\mathbb R$-Cartier. 
Then we can define numerically lc surfaces and numerically dlt surfaces 
(see \cite[Notation 4.1]{kollar-mori}). Precisely speaking, we have: 

\begin{say}[{Numerically lc and dlt due to Koll\'ar--Mori (see \cite[Notation 4.1]{kollar-mori})}]\label{say2.2}
Let $X$ be a normal surface and let $\Delta$ be an $\mathbb R$-divisor 
on $X$ whose coefficients are in $[0, 1]$. 
Let $f:Y\to X$ be a projective birational morphism 
from a smooth variety $Y$ with the exceptional divisor $E=\sum _i E_i$. 
Then the system of linear equations 
$$
E_j\cdot (\sum _i a_i E_i)=E_j \cdot (K_Y+f^{-1}_*\Delta)
$$ 
for any $j$ has a unique solution. 
We write this as 
$$
K_Y+f^{-1}_*\Delta\equiv \sum _ia(E_i, X, \Delta)E_i 
$$ 
with $a(E_i, X, \Delta)=a_i$. In this situation, we say that 
$(X, \Delta)$ is numerically lc if $a(E_i, X, \Delta)\geq -1$ for every exceptional 
curve $E_i$ and every resolution of singularities $f:Y\to X$. 
We say that $(X, \Delta)$ is numerically dlt if there exists a finite 
set $Z\subset X$ such that 
$X\setminus Z$ is smooth, 
$\Supp \Delta|_{X\setminus Z}$ is a simple normal crossing divisor, 
and $a(E, X, \Delta)>-1$ for every 
exceptional curve $E$ which maps to $Z$. 
\end{say}

Let us recall the basic operations and notation 
for $\mathbb R$-divisors. 
 
\begin{say}[$\mathbb R$-divisors]\label{say2.3}
Let $D=\sum a_iD_i$ be an $\mathbb R$-divisor on a normal variety $X$. 
Note that $D_i$ is a prime divisor for every $i$ and 
that $D_i\ne D_j$ for $i\ne j$. 
Of course, $a_i \in \mathbb R$ for every $i$. 
We put $\lfloor D\rfloor= \sum \lfloor a_i \rfloor D_i$ and 
call it the round-down of $D$. 
Note that, for every real number $x$, $\lfloor x\rfloor$ is the integer defined 
by $x-1<\lfloor x\rfloor \leq x$. 
We also put $\lceil D\rceil =-\lfloor -D\rfloor$ and call it the round-up of $D$. 
The fractional part $\{D\}$ denotes $D-\lfloor D\rfloor$. 
We put $$D^{=1}=\sum _{a_i=1} D_i 
\quad \text{and}\quad D^{<1}=\sum _{a_i<1}a_iD_i. $$ 

Let $B_1$ and $B_2$ be two $\mathbb R$-Cartier divisors 
on a normal variety $X$. 
Then $B_1$ is $\mathbb R$-linearly equivalent to 
$B_2$, denoted by $B_1\sim _{\mathbb R}B_2$, if 
$$B_1=B_2+\sum _{i=1}^k r_i (f_i)
$$ 
such that $f_i \in \mathbb C(X)$ and $r_i\in \mathbb R$ 
for every $i$. 
We note that 
$(f_i)$ is a principal Cartier divisor 
associated to $f_i$. 
Let $f:X\to Y$ be a morphism to a variety $Y$. 
If there is an $\mathbb R$-Cartier divisor $B$ on $Y$ such that 
$$
B_1\sim _{\mathbb R}B_2+f^*B, 
$$ 
then $B_1$ is said to be relatively $\mathbb R$-linearly equivalent to 
$B_2$. It is denoted by $B_1\sim _{\mathbb R, f}B_2$ or 
$B_1\sim _{\mathbb R, Y}B_2$. 

\end{say}

\section{Proof of theorems}\label{sec3} 

In this section, we prove Theorem \ref{thm1.1} and Theorem 
\ref{thm1.2}. 
Let us prove Theorem \ref{thm1.1}. 

\begin{proof}[Proof of Theorem \ref{thm1.1}]
By Kodaira type vanishing theorem for log canonical 
pairs (see, for example, \cite[Theorem 8.1]{fujino-fund} and 
\cite[Theorem 5.6.4]{fujino-foundation}), 
we have $R^if_*\mathcal O_X=0$ for every $i>0$. 
Therefore, we have $Rf_*\mathcal O_X\simeq \mathcal O_Y$. 
Then, by Kov\'acs's characterization of rational singularities 
(see \cite[Theorem 1]{kovacs} and \cite[Theorem 3.12.5]{fujino-foundation}), 
we obtain that $Y$ has only rational singularities. 
When $f$ is birational, see also Lemma \ref{lem3.1} below. 
\end{proof}

The following lemma is obvious by the definition of 
rational singularities. 

\begin{lem}\label{lem3.1} 
Let $f:X\to Y$ be a proper birational morphism 
between normal varieties. 
Assume that 
$R^if_*\mathcal O_X=0$ for every $i>0$. 
Then $X$ has only rational singularities 
if and only if $Y$ has only rational singularities.  
\end{lem}

Here, we give a proof of \cite[Theorem 1.2]{alexeev-hacon}, which 
is a main ingredient of Theorem \ref{thm1.2}, for 
the reader's convenience. 

\begin{thm}[{\cite[Theorem 1.2]{alexeev-hacon}}]\label{thm3.2}
Let $(X, \Delta)$ be a log canonical 
pair and let $f:Y\to X$ be a resolution of singularities. 
Then every associated prime of $R^if_*\mathcal O_Y$ is the generic 
point of some 
log canonical center of $(X, \Delta)$ for every $i>0$.  
\end{thm}

Note that $R^if_*\mathcal O_Y$ is independent of the 
resolution $f:Y\to X$. 

\begin{proof}
Without loss of generality, we may assume that 
$X$ is quasi-projective by shrinking $X$. 
We take a dlt blow-up $g: (Z, \Delta_Z)\to (X, \Delta)$ 
(see, for example, \cite[Theorem 4.4.21]{fujino-foundation} and 
\cite[Section 10]{fujino-fund}). 
This means that $g$ is a projective 
birational morphism 
such that $K_Z+\Delta_Z=g^*(K_X+\Delta)$ and that 
$(Z, \Delta_Z)$ is a divisorial log terminal pair. 
It is well known that $Z$ has only rational singularities. 
We take a projective birational morphism 
$h: Y\to Z$ such that 
$K_Y+\Delta_Y=h^*(K_Z+\Delta_Z)$, 
$Y$ is smooth, and $\Supp \Delta_Y$ is a simple normal crossing divisor 
on $Y$. 
We may assume that $h$ is an isomorphism 
over the generic point of any log canonical 
center of $(Z, \Delta_Z)$ by Szab\'o's resolution lemma 
(see, for example, \cite[Remark 2.3.18 and Lemma 2.3.19]
{fujino-foundation}). 
Then we have 
$$
K_Y+\{\Delta_Y\} +\Delta^{=1}_Y+\lfloor \Delta^{<1}_Y\rfloor 
=K_Y+\Delta_Y\sim _{\mathbb R, f} 0, 
$$ 
where $f=g\circ h: Y\to X$. 
We put $E=\lceil -\Delta^{<1}_Y\rceil$. 
Then $E$ is effective, $h$-exceptional, 
and $E\sim _{\mathbb R, f} K_Y+\{\Delta_Y\}+\Delta^{=1}_Y$. 
Therefore, we obtain $Rh_*\mathcal O_Y(E)\simeq 
\mathcal O_Z$ since $R^ih_*\mathcal O_Y(E)=0$ for every 
$i>0$ by the vanishing theorem of Reid--Fukuda type (see, for example, 
\cite[Lemma 6.2]{fujino-fund} and 
\cite[Theorem 3.2.11]{fujino-foundation}) and 
$h_*\mathcal O_Y(E)\simeq \mathcal O_Z$. 
Note that $Rh_*\mathcal O_Y\simeq \mathcal O_Z$ since 
$Z$ has only rational singularities. 
Thus, 
we obtain 
$$
Rf_*\mathcal O_Y(E)  \simeq  Rg_*Rh_* \mathcal O_Y(E)
 \simeq Rg_*\mathcal O_Z 
 \simeq Rg_*Rh_*\mathcal O_Y
 \simeq Rf_*\mathcal O_Y. 
$$
By \cite[Theorem 6.3 (i)]{fujino-fund} 
(see also \cite[Theorem 3.16.3 (i)]{fujino-foundation}), we have that 
every associated prime of $R^if_*\mathcal O_Y(E)\simeq 
R^if_*\mathcal O_Y$ is the generic 
point of some log canonical center of $(X, \Delta)$ for every $i>0$. 
\end{proof}

Let us prove Theorem \ref{thm1.2}. 

\begin{proof}[Proof of Theorem \ref{thm1.2}]
Let $g:Z\to X^+$ be a resolution of singularities. 
Let $\Exc (f^+)$ be the exceptional locus of $f^+: X^+\to Y$. 
By Theorem \ref{thm1.1}, 
we know that $Y$ has only rational singularities. 
Therefore, $X^+\setminus \Exc(f^+)$ has only rational singularities. 
Thus, $\Supp R^ig_*\mathcal O_Z\subset \Exc(f^+)$ for 
every $i>0$. 
By the negativity lemma (see, for example, \cite[Lemma 3.38]{kollar-mori} and 
\cite[Lemma 2.3.27]{fujino-foundation}), 
there are no log canonical centers of $(X^+, \Delta^+)$ contained 
in $\Exc(f^+)$. 
By Theorem \ref{thm3.2}, 
every associated prime of $R^ig_*\mathcal O_Z$ is 
the generic point of some log canonical 
center of $(X^+, \Delta^+)$ for every $i>0$. 
Thus, we have $R^ig_*\mathcal O_Z=0$ for every $i>0$. 
This means that 
$X^+$ has only rational singularities.  
\end{proof}

\section{On log surfaces}\label{sec4}

In this section, we give some 
results on the minimal model program for log canonical surfaces 
(see \cite{fujino-surface}, \cite{fujino-tanaka}, and \cite{tanaka}). 
This section is a supplement to \cite{fujino-surface}. 

The following theorem is the main result of this section. 

\begin{thm}\label{thm4.1}
Let $(X, \Delta)$ be a log canonical 
surface and let $f:X\to Y$ be a projective 
birational morphism 
onto a normal surface $Y$. Assume that 
$-(K_X+\Delta)$ is $f$-ample. 
Then the exceptional locus $\Exc(f)$ of $f$ 
passes through no nonrational 
singular points of $X$. 
In particular, every $f$-exceptional curve is a 
$\mathbb Q$-Cartier divisor. 
Moreover, if the relative Picard number 
$\rho (X/Y)=1$, then $\Exc(f)$ is an irreducible curve 
and $K_Y+\Delta_Y$, where $\Delta_Y=f_*\Delta$, is $\mathbb R$-Cartier. 
\end{thm}

\begin{proof}
By shrinking $Y$, we may assume that 
$f(\Exc(f))=P$ and that $(Y, \Delta_Y)$, where 
$\Delta_Y=f_*\Delta$, is numerically dlt by the negativity lemma 
(see, for example, \cite[Lemma 3.41]{kollar-mori} and 
\cite[Lemma 2.3.25]{fujino-foundation}). 
Therefore, $Y$ has only rational singularities (see \cite[Theorem 4.12]
{kollar-mori}). 
By the Kodaira type vanishing theorem 
as in the proof of Theorem \ref{thm1.1} (see also \cite[Theorem 6.2]{fujino-tanaka}), 
we obtain $R^if_*\mathcal O_X=0$ for every $i>0$. 
Thus, $X$ has only rational singularities 
in a neighborhood of $\Exc(f)$ by Lemma \ref{lem3.1}. 
This means that $X$ is $\mathbb Q$-factorial around $\Exc(f)$ 
(see, for example, \cite[Proposition (17.1)]{lipman} and 
\cite[Proposition 20.2]{tanaka}). Therefore, every $f$-exceptional 
curve is a $\mathbb Q$-Cartier divisor. 
From now on, we assume that $\rho(X/Y)=1$. 
We take an irreducible $f$-exceptional curve $E$. 
Then $E^2<0$ and $E\cdot C<0$ for every $f$-exceptional curve $C$. 
This means that $E=\Exc(f)$. 
We can take a real number $a$ such that 
$(K_X+\Delta+aE)\cdot E=0$. 
Then, by the contraction theorem (see \cite[Theorem 3.2]{fujino-surface} and 
\cite[Theorem 17.1]{tanaka}), we can check 
that $K_Y+\Delta_Y$ is $\mathbb R$-Cartier and 
$K_X+\Delta+aE=f^*(K_Y+\Delta_Y)$. 
\end{proof}

As an easy consequence of Theorem \ref{thm4.1}, 
we have: 

\begin{cor}\label{cor4.2}
In the minimal model program for log canonical 
surfaces, the number of nonrational log canonical singularities 
never decreases.  
\end{cor}

\begin{rem}\label{rem4.3}
Theorem \ref{thm4.1} and Corollary 
\ref{cor4.2} hold true over any algebraically closed field $k$. 
This is because the vanishing theorems 
for birational morphisms from log surfaces hold 
true even when the characteristic of $k$ is positive 
(see, for example, \cite[Theorem 6.2]{fujino-tanaka}). 
\end{rem}

We give an important remark 
on \cite{fujino-surface}. 

\begin{rem}\label{rem4.4}
In \cite{fujino-surface}, we used the fact that 
a numerically lc surface is a log canonical surface (see \cite[Proposition 3.5 (2)]
{fujino-surface}) for the proof of the minimal model program for 
(not necessarily $\mathbb Q$-factorial) log canonical surfaces 
(see \cite[Theorem 3.3]{fujino-surface}). 
Note that the proof of \cite[Proposition 3.5 (2)]{fujino-surface} more or 
less depends on the classification of numerically lc surface singularities 
in \cite[Theorem 4.7]{kollar-mori}. 
By using Theorem \ref{thm4.1}, 
we can check that $K_{X_i}+\Delta_i$ is $\mathbb R$-Cartier 
in the poof of \cite[Theorem 3.3]{fujino-surface} without 
using \cite[Proposition 3.5 (2)]{fujino-surface}. 
This means that the minimal model program 
for log canonical surfaces in \cite[Theorem 3.3]{fujino-surface} is 
independent of the classification of (numerically) lc surface singularities 
(see \cite[Theorem 4.7]{kollar-mori}). 
\end{rem}

\section{Examples}\label{sec5}
 
In this section, let us see that nonrational singularities sometimes 
may cause undesirable phenomena. 

Note that the log canonical model of a log canonical surface 
may have nonrational singularities. 

\begin{ex}\label{ex5.1}
Let $C\subset \mathbb P^2$ be an elliptic curve and 
let $V\subset \mathbb P^3$ be a cone over $C\subset \mathbb P^2$. 
Let $p:X\to V$ be the blow-up at the vertex $P$ of $V$. 
We take a general very ample smooth Cartier divisor 
$\Delta_V$ on $V$ such that 
$K_V+\Delta_V$ is very ample. 
We put $K_X+\Delta=p^*(K_V+\Delta_V)$. 
Then $X$ is smooth, 
$(X, \Delta)$ is log canonical, 
and $K_X+\Delta$ is big. 
Note that $p=\Phi_{|K_X+\Delta|}: X\to V$. 
We also note that $(V, \Delta_V)$ is log canonical 
and that the singularity $P\in V$ is not rational. 
\end{ex}

A finite \'etale morphism 
between kawamata log terminal pairs of 
log general type induces a natural finite \'etale cover 
of their log canonical models in any dimension. 

\begin{thm}\label{thm5.2}
Let $X$ be a normal projective variety 
and let $\Delta$ be an effective $\mathbb Q$-divisor on 
$X$ such that 
$(X, \Delta)$ is kawamata log terminal. 
Let $f:Y\to X$ be a finite \'etale morphism 
such that $K_Y+\Delta_Y=f^*(K_X+\Delta)$. 
Assume that 
$K_X+\Delta$ is big. 
Then we have a commutative 
diagram 
$$
\xymatrix{
Y\ar[d]_{f} \ar@{-->}[r]^q& Y_c\ar[d]^{f_c}\\
X \ar@{-->}[r]_p& X_c 
}
$$
where $p$ and $q$ are birational maps, 
$(X_c, \Delta_c)$ {\em{(}}resp.~$(Y_c, \Delta_{Y_c})${\em{)}} is the 
log canonical model of $(X, \Delta)$ 
{\em{(}}resp.~$(Y, \Delta_Y)${\em{)}}, 
$f_c$ is a finite \'etale morphism, and 
$K_{Y_c}+\Delta_{Y_c}=f_c^*(K_{X_c}+\Delta_c)$. 
\end{thm}
\begin{proof}
The proof of \cite[Theorem 4.5]{fujino-zucker} works with some suitable 
modifications. Note that $X_c$ and $Y_c$ have only rational singularities 
since $(X_c, \Delta_c)$ and $(Y_c, \Delta_{Y_c})$ are both 
kawamata log terminal pairs. 
We leave the details as an exercise for the reader. 
\end{proof}

Unfortunately, Theorem \ref{thm5.2} does not hold for 
log canonical pairs. 
This is because log canonical models of log canonical 
pairs sometimes 
have nonrational singularities. 

\begin{ex}\label{ex5.3}
Let $p: X\to V$ be as in Example \ref{ex5.1} and 
let $E$ be the $p$-exceptional 
divisor on $X$. 
Note that there is a natural $\mathbb P^1$-bundle structure 
$\pi:X\to C$ and $E$ is a section of $\pi$. 
We take a nontrivial finite \'etale cover $D\to C$. 
We put $Y=X\times _C D$ and $F=E\times _C D$. 
We put $K_Y+\Delta_Y=f^*(K_X+\Delta)$. 
Let $W$ be the log canonical model of $(Y, \Delta_Y)$. 
Then we have the following 
commutative diagram 
$$
\xymatrix{
Y \ar[r]^{q}\ar[d]_f & W \ar[d]^{h}\\ 
X\ar[r]_p& V
}
$$
such that $f$ is \'etale, $h$ is finite, but $h$ is not \'etale. 
Note that $q$ contracts $F$ to an isolated 
normal singular point $Q$ of $W$ such 
that $h^{-1}(P)=Q$ since $f^{-1}(E)=F$. 
We also note that 
the singularities of $V$ and $W$ are not rational.   
\end{ex}

\end{document}